\documentclass{article}%
\usepackage{amsfonts, amsmath, amssymb}
\usepackage{subfig}
\usepackage{pstricks, pst-plot, pst-node, multido}
\usepackage{amsmath}
\usepackage{amsfonts}
\usepackage{amssymb}
\usepackage{graphicx}%
\providecommand{\U}[1]{\protect\rule{.1in}{.1in}}
\newtheorem{theorem}{Theorem}

\newtheorem{definition}[theorem]{Definition}

\newtheorem{proposition}[theorem]{Proposition}
\newtheorem{remark}[theorem]{Remark}

\newenvironment{proof}[1][Proof]{\noindent\textbf{#1.} }{\ \rule{0.5em}{0.5em}}
\begin{document}

\title{Bifurcation of periodic solutions from a ring configuration in the vortex and
filament problems}
\author{C. Garc\'{\i}a-Azpeitia
\and J. Ize\\{\small Depto. Matem\'{a}ticas y Mec\'{a}nica, IIMAS-UNAM, FENOMEC, }\\{\small Apdo. Postal 20-726, 01000 M\'{e}xico D.F. }\\{\small cgazpe@hotmail.com}}
\maketitle

\begin{abstract}

This paper gives an analysis of the movement of $n+1$ almost parallel
filaments or vortices. Starting from a polygonal equilibrium of $n$ vortices
with equal circulation and one vortex at the center of the polygon, we find
bifurcation of periodic solutions. The bifurcation result makes use of the
orthogonal degree in order to prove global bifurcation of periodic solutions
depending on the circulation of the central vortex. In the case of the
filament problem these solutions are periodic traveling waves.


MSC 34C25, 37G40, 47H11, 54F45

\end{abstract}

\section{Introduction}

Consider $n$ point vortices or $n$ almost parallel filaments turning at a
constant speed in a plane around some central point. A relative equilibrium of
this configuration is a stationary solution of the equations in the rotating
coordinates. In this paper, we give a complete study of the polygonal relative
equilibrium where there are $n$ identical point vortices arranged on a regular
polygon and a central vortex, with a possibly different circulation. For this
polygonal equilibrium, we give a full analysis for the bifurcation of periodic
solutions. For the filament problem, we prove that there is a global
bifurcation of periodic solutions of waves traveling in the vertical
direction. In this problem, a study of periodic solutions in the vertical
direction would lead to a small divisors setting.

There has been a renewed interest in point vortex problems in the last 30
years, as a model for fluid mechanics. We refer to \cite{Ne01}, for an
up-to-date study, or \cite{MeHa91} for a more general reference. There are
many papers on relative equilibria and some on application of KAM theory but
few on periodic solutions. In the case of nearly parallel filaments, we shall
use the model proposed by Klein, Majda and Damodaran, see \cite{KMD95} and
\cite{KPV03}.

The linearization of the system at a critical point is a ${n}\times{n}$
matrix, which is non invertible, due to the rotational symmetry. These facts
imply that the study of the spectrum of the linearization is not an easy task
and that the classical bifurcation results for periodic solutions may not be
applied directly. However, we shall use the change of variables proved in our
previous paper, \cite{GaIz11}, in order to give not only this spectrum but
also the consequences for the symmetries of the solutions.

The present paper is a continuation of \cite{GaIz11}, where we had a complete
study of the bifurcation of relative equilibria. See also \cite{MeSc88}. Thus,
we shall use the results in that paper, but we shall recall all the important
notions. In a parallel paper, \cite{GaIz112}, we study a similar problem for
point masses. Although there are many similarities, in particular in the
change of variables, the results are of a quite different nature.

The next section is devoted to the mathematical setting of the problem, with
the symmetries involved. Then, we give, in the following two sections, the
preliminary results needed in order to apply the orthogonal degree theory
developed in \cite{IzVi03}, that is the global Liapunov-Schmidt reduction, the
study of the irreducible representations, with the change of variables of
\cite{GaIz11}, and the symmetries associated to these representations. In the
next section, we prove our bifurcation results and, in the following section,
we give the analysis of the spectrum, with the complete results on the type of
solutions which bifurcate from the relative equilibrium, in the vortex and the
filament cases.

\section{Setting the problem}

Let us denote by $q_{j}(t,s)\in\mathbb{R}^{2}$ the position of the $j$'th
filament, where $s$ represents the vertical axis. Let us suppose that one
filament has circulation $\kappa_{0}=\mu$ and $n$ filaments have circulation
$\kappa_{j}=1$, for $j\in\{1,...,n\}$. The dimensionless equations of $n+1$
almost parallel filaments in rotating coordinates, $u_{j}(t)=e^{-\omega
Jt}q_{j}(t)$, are given by
\[
\kappa_{j}J\partial_{t}u_{j}+\kappa_{j}^{2}\partial_{ss}^{2}u_{j}=\omega
\kappa_{j}u_{j}-\sum_{i=0(i\neq j)}^{n}\kappa_{i}\kappa_{j}\frac{u_{j}-u_{i}%
}{\left\Vert u_{j}-u_{i}\right\Vert ^{2}}%
\]
where $J$ is the canonical symplectic matrix.

Define the vector $u=(u_{0},u_{1},...,u_{n})^{T}$, the matrix of circulations
$\mathcal{K}=diag(\mu I,I,...,I)$ and the symplectic matrix $\mathcal{J}%
=diag(J,J,...,J)$. Then, the equations of the filaments, in vectorial form,
can be written as%
\begin{align*}
\mathcal{KJ}u_{t}+\mathcal{K}^{2}u_{ss}  &  =\nabla V(u)\text{ with }\\
V(u)  &  =\omega\frac{1}{2}(u^{T}\mathcal{K}u)-\sum_{i<j}\kappa_{i}\kappa
_{j}\ln(\left\Vert u_{j}-u_{i}\right\Vert )\text{.}%
\end{align*}

The previous equations contain, as a particular case when the filaments
$u(t,s)$ are constant on the vertical axis $s$, the vortex problem given by
the equations%
\begin{equation}
\mathcal{KJ}\dot{u}=\nabla V(u).
\end{equation}
In this paper we analyze two problems: The bifurcation of periodic solutions
for the vortex problem, and bifurcation of traveling waves for the filament
problem. In order to find traveling waves, one sets $u(t,s)=u(\gamma t+s)$,
then the traveling waves become solutions of the ordinary differential
equation%
\begin{equation}
\mathcal{K}^{2}\ddot{u}+\gamma\mathcal{KJ}\dot{u}=\nabla V(u).
\end{equation}

Now, the critical points of the potential $V$ correspond to relative
equilibria of the problem. Actually, the polygonal configuration
\[
\bar{a}=(0,e^{i\zeta},...,e^{in\zeta})
\]
is a relative equilibrium when $\omega=s_{1}+\mu$, with $s_{1}=(n-1)/2$. This
fact is proven in \cite{GaIz11}, where the bifurcation of relative equilibria
is analyzed using $\mu$ as a parameter.

\begin{remark}
If, as we have done in \cite{GaIz112}, one replaces, in the change of
coordinates, the term $e^{\omega tJ}$ with a complex factor $\varphi(t)$
(taking $q_{j}$ and $u_{j}$ as complex functions instead of a planar vector),
where $\varphi$ satisfies the equation%
\[
i{\varphi}^{\prime}=-\omega\varphi/{\arrowvert\varphi\arrowvert}^{2}\text{,}%
\]
then the equations, for the vortex problem, become%

\[
{\arrowvert\varphi\arrowvert}^{2}\mathcal{K}i\dot{u}=\nabla V(u)\text{.}%
\]
In particular, the stationary solutions of this system are the same solutions
studied in \cite{GaIz11}. However, the solutions of the equation for $\varphi$
are $ce^{i\nu t}$, with $\nu=\omega/\arrowvert c^{2}\arrowvert$, that is
circular orbits only. Thus, the vortex problem differs from the masses
problem, where a similar argument gives all possible conical orbits.
\end{remark}

In this paper we shall prove bifurcation of periodic solutions from the
polygonal equilibrium $\bar{a}$. This is an analogous treatment to the
bifurcation of periodic solutions for the $(n+1)$-body problem in
\cite{GaIz112}.

Changing variables by $x(t)=u(t/\nu)$, the $2\pi/\nu$-periodic solutions of
the differential equation for the filament become zeros of the bifurcation
operator%
\begin{align*}
f  &  :H_{2\pi}^{2}(\mathbb{R}^{2(n+1)}\backslash\Psi)\rightarrow L_{2\pi}%
^{2}\\
f(x,\nu)  &  =-\nu^{2}\mathcal{K}^{2}\ddot{x}-\gamma\nu\mathcal{KJ}\dot
{x}+\nabla V(x)\text{,}%
\end{align*}
where the set $\Psi=\{x\in\mathbb{R}^{2(n+1)}:x_{i}=x_{j}\}$ consists of the
collision points, and the set
\[
H_{2\pi}^{2}(\mathbb{R}^{2(n+1)}\backslash\Psi)=\{x\in H_{2\pi}^{2}%
(\mathbb{R}^{2(n+1)}):x_{i}(t)\neq x_{j}(t)\}
\]
consists of the collision-free orbits.

\begin{remark}
We shall concentrate the analysis on the filament problem. All the following
statements apply to the vortex problem, except that one has the bifurcation
operator $f(x)\mathcal{=-}\nu\mathcal{KJ}\dot{x}+\nabla V(x)$ defined in the
space $H_{2\pi}^{1}(\mathbb{R}^{2(n+1)}\backslash\Psi)$.
\end{remark}

\begin{definition}
Let $S_{n}$ be the group of permutations of $\{1,...,n\}$. One defines the
action of $S_{n}$ in $\mathbb{R}^{2(n+1)}$ as%
\[
\rho(\gamma)(x_{0},x_{1},...,x_{n})=(x_{0},x_{\gamma(1)},...,x_{\gamma(n)}).
\]

\end{definition}

The gradient $\nabla V$ is $S_{n}$-equivariant, that is it commutes with the
action of the group, because $n$ vortices have the same circulation. Let
$\mathbb{Z}_{n}$ be the subgroup of permutations generated by $\zeta(j)=j+1$
modulus $n$, then the map $f$ is $\Gamma\times S^{1}$-equivariant with the
abelian group%
\[
\Gamma=\mathbb{Z}_{n}\times S^{1}.
\]

Now, the infinitesimal generators of $S^{1}$ and $\Gamma$ are%
\[
\text{ }Ax=\frac{\partial}{\partial\varphi}|_{\varphi=0}x(t+\varphi)=\dot
{x}\text{ and }A_{1}x=\frac{\partial}{\partial\theta}|_{\theta=0}%
e^{-\mathcal{J\theta}}x=-\mathcal{J}x\text{.}%
\]
Since $V$ is $\Gamma$-invariant, then the gradient $\nabla V(x)$ must be
orthogonal to the generator $A_{1}x$. As a consequence, the map $f$ must be
$\Gamma\times S^{1}$-orthogonal, due to the equalities%
\begin{align*}
\left\langle f(x),\dot{x}\right\rangle _{L_{2\pi}^{2}}  &  =-\frac{\nu^{2}}%
{2}\left\Vert \mathcal{K}\dot{x}\right\Vert ^{2}|_{0}^{2\pi}-\gamma\nu
\sum\kappa_{j}\left\langle \mathcal{J}\dot{x}_{j},\dot{x}_{j}\right\rangle
_{L_{2\pi}^{2}}+V(x)|_{0}^{2\pi}=0\text{,}\\
\left\langle f(x),\mathcal{J}x\right\rangle _{L_{2\pi}^{2}}  &  =\nu
^{2}\left\langle \mathcal{K}\dot{x},\mathcal{JK}\dot{x}\right\rangle
_{L_{2\pi}^{2}}-\frac{\gamma\nu}{2}\sum\kappa_{j}\left\Vert x_{j}\right\Vert
^{2}|_{0}^{2\pi}+\int_{0}^{2\pi}\left\langle \nabla V,\mathcal{J}%
x\right\rangle =0\text{.}%
\end{align*}

Define $\mathbb{\tilde{Z}}_{n}$ as the subgroup of $\Gamma$ generated by
$(\zeta,\zeta)\in\mathbb{Z}_{n}\times S^{1}$ with $\zeta=2\pi/n\in S^{1}$.
Since the action of $(\zeta,\zeta)$ leaves fixed the equilibrium $\bar{a}$,
then the isotropy group of $\bar{a}$ is the group ${\Gamma}_{\bar{a}}\times
S^{1}$ with%
\[
\Gamma_{\bar{a}}=\mathbb{\tilde{Z}}_{n}.
\]
Thus, the orbit of $\bar{a}$ is isomorphic to the group $S^{1}$. In fact, the
orbit consists of the rotations of the equilibrium. As a consequence, the
generator of the orbit $A_{1}\bar{a}=-\mathcal{J}\bar{a}$ must be in the
kernel of $D^{2}f(\bar{a})$.

\section{The Liapunov-Schmidt reduction}

In order to apply the orthogonal degree of \cite{IzVi03}, one needs to make a
reduction of the bifurcation map to some finite space.

The bifurcation map $f$ has Fourier series%
\[
f(x)=\sum_{l\in\mathbb{Z}}(l^{2}\nu^{2}\mathcal{K}^{2}x_{l}-\gamma
l\nu(i\mathcal{J)K}x_{l}+g_{l})e^{ilt}\text{,}%
\]
where $x_{l}$ and $g_{l}$ are the Fourier modes of $x$ and $\nabla V(x)$.
Since $l^{2}\nu^{2}\mathcal{K}^{2}-\gamma il\nu(i\mathcal{J)K}$ is invertible
for all big $l$'s, then one may solve $x_{l}$ for $\left\vert l\right\vert >p$
from
\[
l^{2}\nu^{2}\mathcal{K}^{2}x_{l}-\gamma l\nu(i\mathcal{J)K}x_{l}%
+g_{l}=0\text{.}%
\]
Actually, one may use the global implicit function theorem to perform the
reduction globally.

In this way, the bifurcation operator $f$ has the same zeros as the
bifurcation function
\[
f(x_{1},x_{2}(x_{1},\nu),\nu)=\sum_{\left\vert l\right\vert \leq p}(l^{2}%
\nu^{2}\mathcal{K}^{2}x_{l}-\gamma l\nu(i\mathcal{J)K}x_{l}+g_{l}%
)e^{ilt}\text{,}%
\]
and the linearization of the bifurcation function at some equilibrium $\bar
{a}$ is%
\[
f^{\prime}(\bar{a})x_{1}=\sum_{\left\vert l\right\vert \leq p}\left(  l^{2}%
\nu^{2}\mathcal{K}^{2}-\gamma l\nu(i\mathcal{J)K}+D^{2}V(\bar{a})\right)
x_{l}e^{ilt}\text{.}%
\]

Since the bifurcation operator is real, the linearization of the bifurcation
function is determined by blocks $M(l\nu)$ for $l\in\{0,...,p\}$, where
$M(\nu)$ is the matrix%

\[
M(\nu)=\nu^{2}\mathcal{K}^{2}-\gamma\nu(i\mathcal{J)K}+D^{2}V_{\alpha}(\bar
{a})\text{.}%
\]
These blocks $M(l\nu)$ represent the Fourier modes of the linearized equation
at the equilibrium.

\begin{remark}
For the vortex problem one may prove a similar statement and get the Fourier
modes%
\[
M(\nu)=-\nu(i\mathcal{J)K}+D^{2}V(\bar{a}).
\]

\end{remark}

\begin{remark}
Actually, the orthogonal degree could be used to analyze the bifurcation of
periodic solutions, in time and spatial $z$-coordinate, for the filaments.
However, the Liapunov-Schmidt reduction cannot be performed without solving a
small divisor problem.
\end{remark}

\section{Irreducible representations}

In order to apply the orthogonal degree, one needs to find the irreducible
representation subspaces for the action of $\Gamma_{\bar{a}}=\mathbb{\tilde
{Z}}_{n}$.

In the following sections we shall assume that $n>2$, since the case $n=2$ is
different and will be treated at the end of the paper.

\begin{definition}
For $k\in\{2,...,n-2,n\}$, we define the isomorphisms $T_{k}:\mathbb{C}%
^{2}\rightarrow W_{k}$ as%
\begin{align*}
T_{k}(w)  &  =(0,n^{-1/2}e^{(ikI+J)\zeta}w,...,n^{-1/2}e^{n(ikI+J)\zeta
}w)\text{ with}\\
W_{k}  &  =\{(0,e^{(ikI+J)\zeta}w,...,e^{n(ikI+J)\zeta}w):w\in\mathbb{C}%
^{2}\}\text{.}%
\end{align*}
For $k\in\{1,n-1\}$, we define the isomorphism $T_{k}:$ $\mathbb{C}%
^{3}\rightarrow W_{k}$ as
\begin{align*}
T_{k}(\alpha,w)  &  =(v_{k}\alpha,n^{-1/2}e^{(ikI+J)\zeta}w,...,n^{-1/2}%
e^{n(ikI+J)\zeta}w)\text{ with}\\
W_{k}  &  =\{(v_{k}\alpha,e^{(ikI+J)\zeta}w,...,e^{n(ikI+J)\zeta}w):\alpha
\in\mathbb{C},w\in\mathbb{C}^{2}\}\text{,}%
\end{align*}
where $v_{1}$ and $v_{n-1}$ are the vectors%
\[
v_{1}=2^{-1/2}\left(  1,i\right)  \text{ and }v_{n-1}=2^{-1/2}\left(
1,-i\right)  .
\]

\end{definition}

In the paper \cite{GaIz11}, we have proven that the subspaces $W_{k}$ are
irreducible representations. Also, we showed that the action of $(\zeta
,\zeta)\in\mathbb{\tilde{Z}}_{n}$ on the space $W_{k}$ is given by
\[
\rho(\zeta,\zeta)=e^{ik\zeta}\text{.}%
\]

Since the subspaces $W_{k}$ are orthogonal, then the linear map%
\[
Pw=\sum_{j=1}^{n}T(w_{k})
\]
is orthogonal, where $w=(w_{1},...,w_{n})$, with $w_{k}\in\mathbb{C}^{3}$ for
$k=1,n-1$ and $w_{k}\in\mathbb{C}^{2}$ for the other $k$'s.

Since the map $P$ rearranges the coordinates of the irreducible
representations, one has, from Schur's lemma, that%
\[
P^{-1}D^{2}V(\bar{a})P=diag(B_{1},...,B_{n})\text{,}%
\]
where $B_{k}$ are matrices which satisfy $D^{2}V(\bar{a})T_{k}(w)=T_{k}%
(B_{k}w)$. In the paper \cite{GaIz11}, we have found the blocks $B_{k}$: they
satisfy $B_{n-k}=\bar{B}_{k}$ and
\begin{align*}
B_{k}  &  =diag\left(  2(\mu+s_{1})-s_{k},s_{k}\right)  \text{ for }%
k\in\{2,...,n-2,n\}\text{, and}\\
B_{1}  &  =\left(
\begin{array}
[c]{ccc}%
\mu\left(  s_{1}+\mu\right)  & -\left(  n/2\right)  ^{1/2}\mu & -\left(
n/2\right)  ^{1/2}\mu i\\
-\left(  n/2\right)  ^{1/2}\mu & s_{1}+2\mu & 0\\
\left(  n/2\right)  ^{1/2}\mu i & 0 & s_{1}%
\end{array}
\right)  \text{,}%
\end{align*}
where%
\[
s_{k}=k(n-k)/2.
\]

For the linearization of the equation one has that%
\[
P^{-1}M(\nu)P=diag(m_{1}(\nu),..., m_{n}(\nu))\text{.}
\]

Thus, we find the matrices $m_{k}(\nu)$ in terms of the blocks $B_{k}$.

\begin{proposition}
The matrices $m_{k}(\nu)$ satisfy $m_{k}(\nu)=\bar{m}_{n-k}(-\nu)$ with%
\begin{align*}
m_{k}(\nu)  &  =\nu^{2}I-2\gamma\nu(iJ)+B_{k}\text{ for }k\in
\{2,...,n-2,n\}\text{, and}\\
m_{1}(\nu)  &  =\nu^{2}diag(\mu^{2},I)-2\gamma\nu diag(\mu,iJ)+B_{1}\text{.}%
\end{align*}

\end{proposition}

\begin{proof}
For $k\in\{2,...,n-2,n\}$, the matrix $\mathcal{K}$\ in the space $W_{k}$ is
$\mathcal{K}T_{k}(z)=T_{k}(z)$, and the matrix $i\mathcal{J}$ satisfies
$\mathcal{J}T_{k}(z)=T_{k}(Jz)$ . The matrix $\mathcal{K}$ in $W_{1}$ is
$\mathcal{K}T_{1}(z)=T_{1}(diag(\mu,1,1)z)$, and since $(iJ)v_{1}=v_{1}$, then
the matrix $i\mathcal{J}$ satisfies $(i\mathcal{J})T_{1}(z)=T_{1}%
(diag(1,iJ)z)$. From these facts we conclude the statements.

Finally, using that $(iJ)v_{2}=-v_{2}$, one has
\[
m_{n-1}(\nu)=\nu^{2}diag(\mu^{2},I)-2\gamma\nu diag(-\mu,iJ)+B_{n-1}\text{.}%
\]
Then, from the equality $B_{n-k}=\bar{B}_{k}$, one obtains the equality
$m_{n-k}(\nu)=\bar{m}_{k}(-\nu)$.
\end{proof}

\begin{remark}
Analogously, for the vortex problem one has that $m_{k}(\nu)=\bar{m}%
_{n-k}(-\nu)$ with%
\begin{align*}
m_{k}(\nu)  &  =-\nu(iJ)+B_{k}\text{ for }k\in\{2,...,n-2,n\}\text{, and}\\
m_{1}(\nu)  &  =-\nu diag(\mu,iJ)+B_{1}\text{.}%
\end{align*}

\end{remark}

The action of $(\zeta,\zeta,\varphi)\in\mathbb{\tilde{Z}}_{n}\times S^{1}$ on
$W_{k}$ is $\rho(\zeta,\zeta,\varphi)=e^{ik\zeta}e^{il\varphi}$. Therefore,
the isotropy group of the space $W_{k}$ is
\[
\mathbb{Z}_{n}(k)=\left\langle \left(  \zeta,\zeta,-k\zeta\right)
\right\rangle .
\]

\section{Bifurcation theorem}

The fixed point subspace of the isotropy group $\Gamma_{\bar{a}}\times S^{1}$
corresponds to the block $m_{n}(0)=B_{n}$. Since the generator of the kernel
is $A_{1}\bar{a}=T_{n}(-n^{1/2}e_{2})$, then $e_{2}$ must be in the kernel of
$m_{n}(0)$.

Following \cite{IzVi03}, one defines $\sigma$ to be the sign of $m_{n}(0)$ in
the orthogonal subspace to $e_{2}$. Since $B_{n}=2diag\left(  \omega,0\right)
$, then%
\[
\sigma=sgn(e_{1}^{T}B_{n}e_{1})=sgn(\omega).
\]

We have proven, in \cite{GaIz11}, that $m_{k}(0)=B_{k}$ is invertible except
for a point $\mu_{k}$ for $k=1,..,n-1$: $\mu_{k}=s_{k}/2 -s_{1}$ for
$k=2,...,[n/2]$ and $\mu_{1}=s_{1}^{2}$. Therefore, the hypotheses of
\cite{IzVi03} apply for $\mu\neq$ $\mu_{k}$. In fact, in \cite{GaIz11}, we
proved a global bifurcation of stationary solutions from each $\mu_{k}$. See
also \cite{MeSc88}.

\begin{definition}
Following \cite{IzVi03}, we define%
\[
\eta_{k}(\nu_{0})=\sigma\{n_{k}(\nu_{0}-\rho))-n_{k}(\nu_{0}+\rho)\}\text{,}%
\]
where $n_{k}(\nu)$ is the Morse index of $m_{k}(\nu)$.
\end{definition}

This number corresponds to the jump of the orthogonal index at $\nu_{0}$.
Then, from the results of \cite{IzVi03}, we can state the following theorem
for $\mu\neq\mu_{1},...,\mu_{n-1}$.

The orthogonal degree is defined for orthogonal maps that are non-zero on the
boundary of some open bounded invariant set. The degree is made of integers,
one for each orbit type, and it has all the properties of the usual Brouwer
degree. Hence, if one of the integers is non-zero, then the map has a zero
corresponding to the orbit type of that integer. In addition, the degree is
invariant under orthogonal deformations that are non-zero on the boundary. The
degree has other properties such as sum, products and suspensions, for
instance, the degree of two pieces of the set is the sum of the degrees.

Now, if one has an isolated orbit, then its linearization at one point of the
orbit $x_{0}$ has a block diagonal structure, due to Schur's lemma, where the
isotropy subgroup of $x_{0}$ acts as $\mathbb{Z}_{n}$ or as $S^{1}$.
Therefore, the orthogonal index of the orbit is given by the signs of the
determinants of the submatrices where the action is as $\mathbb{Z}_{n}$, for
$n=1$ and $n=2$, and the Morse indices of the submatrices where the action is
as $S^{1}$. In particular, for problems with a parameter, if the orthogonal
index changes at some value of the parameter, one will have bifurcation of
solutions with the corresponding orbit type. Here, the parameter is the
frequency $\nu$.

\begin{theorem}
If $\eta_{k}(\nu_{k})$ is different from zero, then the polygonal equilibrium
has a global bifurcation of periodic solutions from $2\pi/\nu_{k}$ with
isotropy group $\mathbb{\tilde{Z}}_{n}(k)$.
\end{theorem}

The solutions with isotropy group $\mathbb{\tilde{Z}}_{n}(k)$ must satisfy the
symmetries%
\[
u_{j}(t)=e^{-i\zeta}u_{\zeta(j)}(t-k\zeta).
\]
For the $n$ elements with equal circulation,if we use the notation
$u_{j}=u_{j+kn}$ for $j\in\{1,...,n\}$, then $\zeta(j)=j+1$, and the $n$
elements satisfy%
\[
u_{j+1}(t)=e^{\ ij\zeta}u_{1}(t+jk\zeta).
\]
On the other hand, the central element remains at the origin if $k$ and $n$
have a common factor and, if they are relatively prime, then $u_{0}%
(t)=e^{ik^{-1}\zeta}u_{0}(t+\zeta)$, where $k^{-1}$ is such that $k^{-1}k=1$,
modulo $n$.

By global bifurcation, we mean that the branch goes to infinity in norm or
period, or goes to the collision set (in these three cases, we say that the
bifurcation is non admissible) or, if none of the above happens, then the sum
of the above jumps, over all the bifurcation points, is zero. See \cite{Iz95}
for this kind of arguments.

A complete description of these solutions may be found in the paper
\cite{GaIz112} and in \cite{Ga10}.

\section{Spectral analysis}

The $m_{k}(\nu)$ have real eigenvalues, so the matrices $m_{n-k}(\nu)$ and
$m_{k}(-\nu)$ have the same spectrum due to the equality $m_{n-k}(\nu)=\bar
{m}_{k}(-\nu)$. As a consequence, the Morse numbers satisfy
\[
n_{n-k}(\nu)=n_{k}(-\nu).
\]

In order to analyze easily the spectrum of these problems, one may use the
parameter $\mu$ or equivalently $\omega=\mu+s_{1}$.

\subsection{Vortex problem}

\subsubsection{Blocks $k\in\{2,...,n-2,n\}$}

In this case, the blocks are given by
\[
B_{k}=2diag\left(  \omega-\omega_{k},\omega_{k}\right)  \text{,}%
\]
with $\omega_{k}=s_{k}/2$.

\begin{proposition}
Define $\nu_{k}$ as%
\[
\nu_{k}=[4\omega_{k}(\omega-\omega_{k})]^{1/2}\text{.}%
\]
For $\omega>\omega_{k}$, the matrix $m_{k}(\nu)$ changes its Morse index at
the positive value $\nu_{k}$ with%
\[
\eta_{k}(\nu_{k})=-1.
\]
For $\omega<\omega_{k}$, the matrix $m_{k}(\nu)$ is always invertible.
\end{proposition}

\begin{proof}
The determinant of $m_{k}(\nu)$ is $d_{k}(\nu)=-\nu^{2}+4\omega_{k}%
(\omega-\omega_{k})$, and it changes sign only at $\pm\nu_{k}$ for
$\omega>\omega_{k}$. The trace of $m_{k}(\nu)$ is $T_{k}(\nu)=2\omega$, and
$\omega_{k}$ is positive, then $T_{k}(0)>0$ for $\omega>\omega_{k}$. As a
consequence, one has that $n_{k}(0)=0$ and $n_{k}(\infty)=1$, because
$d_{k}(0)>0$ and $d_{k}(\infty) < 0$. Therefore, $\eta(\nu_{k})=0-1$, since
$\sigma=1$, for $\omega>\omega_{k}$.
\end{proof}

Since $\omega_{n}=0$, then the value $\nu_{n}=0$ is not a bifurcation point
for $k=n$.

\begin{theorem}
For each $k\in\{2,...,n-2\}$ and $\omega>\omega_{k}$, the polygonal
equilibrium has a global bifurcation of periodic solutions from $2\pi/\nu_{k}%
$, with symmetries $\mathbb{\tilde{Z}}_{n}(k)$.
\end{theorem}

Sine the bifurcation points have the same sign $\eta_{k}=-1$, then these
branches cannot return to the same equilibrium. Therefore, the bifurcation
branch is non- admissible or goes to another equilibrium.

\subsubsection*{Blocks $k\in\{1,n-1\}$}

Due to the equality $n_{n-1}(\nu)=n_{1}(-\nu)$, one may analyze the spectrum
of the block $m_{1}(\nu)$ in $\mathbb{R}^{+}$, instead of the spectrum of
$m_{n-1}(\nu)$ in $\mathbb{R}^{+}$.%
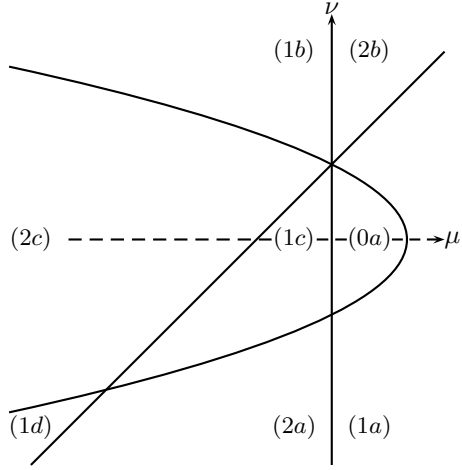
\begin{figure}[ht]
\centering\subfloat{\begin{pspicture}(-4.5,-3)(1.5,3)\SpecialCoor
\psline[linestyle=dashed]{->}(.8,0)(1.5,0)
\psline[linestyle=dashed](-3.5,0)(-.8,0)
\psline[linestyle=dashed](-.2,0)(.2,0)
\rput[l](1.5,0){$\mu$}
\rput[b](0,3){$\nu$}
\psline{->}(0,-3)(0,3)
\psline(-4,-3)(1.5,2.5)
\rput{90}{\parabola(2.3,4.29)(0,-1)}
\rput(.5,2.5){\small$(2b)$}
\rput(-.5,-2.5){\small$(2a)$}
\rput(-.5,2.5){\small$(1b)$}
\rput(.5,-2.5){\small$(1a)$}
\rput(-4,0){\small$(2c)$}
\rput(-.5,0){\small$(1c)$}
\rput(-4,-2.5){\small$(1d)$}
\rput(.5,0){\small$(0a)$}
\NormalCoor\end{pspicture}}
\caption{Graph $d_{1}(\mu,\nu)=0$.}
\end{figure}%

\begin{proposition}
The matrix $m_{1}(\nu)$ changes its Morse index only at the curves $\mu=0$ for
$\nu\in\mathbb{R}$, $\nu_{0}(\mu)=\mu+s_{1}$ for $\mu\in\mathbb{R}$, and
\[
\nu_{\pm}(\mu)=\pm\sqrt{s_{1}^{2}-\mu}\text{ for }\mu\in(-\infty,s_{1}%
^{2})\text{.}%
\]
Moreover, the Morse number of $m_{1}(\nu)$ in the eight regions are: $n_{1}=0$
in the regions (0x), $n_{1}=1$ in the regions (1x), $n_{1}=2$ in the regions (2x).
\end{proposition}

\begin{proof}
The block $m_{1}(\nu)$ is
\[
m_{1}(\nu)=\left(
\begin{array}
[c]{ccc}%
\mu\left(  -\nu+s_{1}+\mu\right)  & -\left(  n/2\right)  ^{1/2}\mu & -\left(
n/2\right)  ^{1/2}\mu i\\
-\left(  n/2\right)  ^{1/2}\mu & s_{1}+2\mu & i\nu\\
\left(  n/2\right)  ^{1/2}\mu i & -i\nu & s_{1}%
\end{array}
\right)  \text{.}%
\]
Since $n=2s_{1}+1$, then the determinant is%
\[
d_{1}(\nu)=\mu\left(  \nu-(\mu+s_{1})\right)  \left(  \nu^{2}-(s_{1}^{2}%
-\mu)\right)  \text{.}%
\]
Thus, the matrix $m_{1}(\nu)$ changes its Morse index only at $\mu=0$,
$\nu_{0}(\mu)$ and $\nu_{\pm}(\mu)$. Moreover, the curves $\nu_{0}(\mu)$ and
$\nu_{-}(\mu)$ intersect at the point%
\[
(\mu_{0},\nu_{0})=(-2s_{1}-1,-s_{1}-1).
\]
Therefore, the plane $(\mu,\nu)$ is divided in eight regions as shown in the graph.

For $\nu$ big enough, the Morse index of the matrices $m_{1}(\nu)$ and $-\nu
diag(\mu,iJ)$ are the same. Thus, the Morse indices are $n_{1}(\infty)=2$ and
$n_{1}(-\infty)=1$ for $\mu>0$, and $n_{1}(\infty)=1$ and $n_{1}(-\infty)=2$
for $\mu<0$. Therefore, one gets that $n_{1}=2$ in the regions (2a) and (2b),
and $n_{1}=1$ in the regions (1a) and (1b).

The region (1d) lies between the curves $\nu_{0}(\mu)$ and $\nu_{-}(\mu)$ for
$\mu\in(-\infty,-2s_{1}-1)$. Since the determinant $d_{1}$ is negative in
(1d), then the Morse index satisfies $n_{1}\in\{1,3\}$. Now, the trace of
$m_{1}(\nu)$ is
\[
T_{1}(\mu,\nu)=2\mu+2s_{1}+\mu\left(  \mu+s_{1}-\nu\right)  \text{.}%
\]
Setting $\bar{\nu}=(\nu_{0}+\nu_{-})/2$, one has that $\bar{\nu}=\mu/2+o(\mu)$
and $T_{1}(\bar{\nu})=\mu^{2}/2+o(\mu^{2})$ when $\mu\rightarrow\infty$. Since
$T_{1}(\bar{\nu})$ is positive when $\mu\rightarrow-\infty$, then $n_{1}\neq
3$. Therefore, the Morse index must be $n_{1}=1$ in the region (1d).

For $\nu=0$, the determinant is%
\[
d_{1}(\mu)=-\mu(\mu+s_{1})\left(  \mu-s_{1}^{2}\right)  \text{.}%
\]
Hence, $d_{1}>0$ in the region (2c), $d_{1}<0$ in the region (1c), and
$d_{1}>0$ in the region (0a). We conclude that $n_{1}\in\{1,3\}$ in (1c) and
$n_{1}\in\{0,2\}$ in (0a). Since $m_{1}(\mu,\nu)$ is continuous and
$m_{1}(0,0)=diag(0,s_{1},s_{1})$, then $n_{1}\leq1$ for $(\mu,\nu)$ near
$(0,0)$. Thus, one has that $n_{1}=1$ in the region (1c), and $n_{1}=0$ in the
region (0a).

Since $d_{1}>0$ in (2c), then $n_{1}\in\{0,2\}$. At the point $(\mu_{0}%
,\nu_{0})$ the trace $T_{1}=-2(s_{1}+1)$ is negative. Since $T_{1}(\mu,\nu)$
is continuous, then $T_{1}<0$ for $(\nu,\mu)$ near $(\mu_{0},\nu_{0})$. Thus,
$n_{1}\neq0$, and the Morse index must be $n_{1}=2$ in the region (2c).
\end{proof}

In the previous section we have found that $n_{1}(\mu,\nu)$ changes its Morse
index on the curves $\nu_{\ast}(\mu)$, for $\ast\in\{0,+,-\}$. Using the
equality $n_{n-1}(\nu)=n_{1}(-\nu)$, we get the following result.

\begin{theorem}
For each $k\in\{1,n-1\}$ such that $\mu\in(-\infty,s_{1}^{2})$, the polygonal
relative equilibrium has a global bifurcation of periodic solutions starting
from $2\pi/\nu_{+}$ with symmetries $\mathbb{\tilde{Z}}_{n}(k)$. Moreover, for
$k=1$ with $\mu\in(-s_{1},\infty)$ and for $k=n-1$ with $\mu\in(-\infty
,-s_{1})$, there is another bifurcation of periodic solutions with symmetries
$\mathbb{\tilde{Z}}_{n}(k)$.
\end{theorem}

Actually, one may find all the numbers $\eta_{k}(\nu)$, for $k\in\{1,n-1\}$,
using that $\sigma=sgn(\omega)$. Thus, one may see that the bifurcations with
symmetries $\mathbb{\tilde{Z}}_{n}(1)$ have all the same index for $\mu
\in(-\infty,-s_{1})\cup(s_{1}^{2},\infty)$, and also with the symmetries
$\mathbb{\tilde{Z}}_{n}(n-1)$ for $\mu\in(-s_{1},s_{1}^{2})$. Therefore, these
bifurcating branches cannot return to the equilibrium $\bar{a}$, and must be
non- admissible or go to another equilibrium.

For $\mu=0$, the two blocks are
\[
m_{1}(\nu)=\left(
\begin{array}
[c]{cc}%
s_{1} & i\nu\\
-i\nu & s_{1}%
\end{array}
\right)  =m_{n-1}(\nu)\text{.}%
\]
Thus, the determinant $\det m_{1}(\nu)=s_{1}^{2}-\nu^{2}$ is zero at $\pm
s_{1}$ with $\eta_{1}(s_{1})=\eta_{n-1}(s_{1})=-1$. Therefore, there is only
one global bifurcation of periodic solutions from $2\pi/s_{1}$ with symmetries
$\mathbb{\tilde{Z}}_{n}(k)$ for $k\in\{1,n-1\}$. The branches are
non-admissible or go to another equilibrium.

\begin{remark}
We have proven that $m_{k}(\nu)$, for $k\in\{2,...,n-2\}$, is non-invertible
at two points only if $\mu>\mu_{k}$. Also, we proved that $m_{k}(\nu)$, for
$k\in\{1,n-1\}$, is non-invertible at three points only if $\mu<\mu_{1}$.
Moreover, the determinant of $m_{n}(\nu)$ has a double zero at $\nu=0$, due to
the symmetries. Since the $\mu_{k}$'s are increasing for $k\in\{2,...,[n/2]\}$%
, the determinant of the matrix $M(\nu)$ has $2(n+1)$ zeros counted with
multiplicity for%
\[
\mu\in(\mu_{\lbrack n/2]},\mu_{1})\text{.}%
\]

From the previous fact, one may conclude, analogously to the $n$-body problem,
\cite{GaIz112}, that the polygonal relative equilibrium is spectrally stable
only if $\mu\in(\mu_{\lbrack n/2]},\mu_{1})$, where $\mu_{1}=(n-1)^{2}/4$ and
$\mu_{k}=\left(  -k^{2}+nk-2n+2\right)  /4$. This fact is proved in the paper
\cite{CaSc00}.
\end{remark}

\subsection{Filaments}

Next we wish to analyze the spectrum of traveling waves in the filaments. We
have two parameters: the traveling wave velocity $\gamma$ and the frequency
$\omega$.

\subsubsection*{Blocks $k\in\{2,...,n-2,n\}$}

\begin{proposition}
Define $\nu_{+}$\ and $\nu_{-}$\ as%
\[
\nu_{\pm}=\left[  (2\gamma^{2}-\omega)\pm\sqrt{(2\gamma^{2}-\omega
)^{2}-4\omega_{k}(\omega-\omega_{k})}\right]  ^{1/2}\text{.}%
\]
For $\omega<\omega_{k}$, and any $\gamma\in\mathbb{R}$, the matrix $m_{k}%
(\nu)$ changes its Morse index at the positive value $\nu_{+}$, with%
\[
\eta_{k}(\nu_{+})=sgn(\omega)\emph{.}%
\]
For $\omega>\omega_{k}$, and any $\gamma^{2}>\omega/2+(\omega_{k}%
(\omega-\omega_{k}))^{1/2}$, the matrix $m_{k}(\nu)$ changes its Morse index
at the positive values $\nu_{\pm}$, with%
\[
\eta_{k}(\nu_{\pm})=\pm1\emph{.}%
\]

\end{proposition}

\begin{proof}
The block $m_{k}(\nu)$ is given by%
\[
m_{k}(\nu)=\nu^{2}-2\gamma\nu iJ+2diag\left(  \omega-\omega_{k},\omega
_{k}\right)  .
\]
Thus, the trace is $T_{k}(\mu)=2(\nu^{2}+\omega)$ and the determinant is%
\begin{align*}
d_{k}(\nu)  &  =\nu^{4}-2(2\gamma^{2}-\omega)\nu^{2}+4\omega_{k}(\omega
-\omega_{k})\\
&  =(\nu^{2}-\nu_{+}^{2})(\nu^{2}-\nu_{-}^{2}).
\end{align*}
Therefore, the determinant is zero at the four roots $\pm\nu_{\pm}$.

For $\omega<\omega_{k}$ and $\gamma\in\mathbb{R}$, only the value $\nu_{+}$ is
positive. Since $d_{k}(0)=4\omega_{k}(\omega-\omega_{k})<0$, then $n_{k}%
(0)=1$.\ Since all eigenvalues of $m_{k}(\nu)$ are positive for big $\nu$,
then $n(\infty)=0$. Hence, the change in the Morse index is $\eta(\nu
_{+})=\sigma(1-0)$, with $\sigma=sgn(\omega).$

For $\omega>\omega_{k}$ and $\gamma^{2}>\omega/2+(\omega_{k}(\omega-\omega
_{k}))^{1/2}$, the two values $\nu_{\pm}$ are positive. Since $\omega_{k}$ is
positive, then $\sigma=sgn(\omega)=1$. Since $d_{k}(0)>0$ and $T_{k}(0)>0$,
then $n_{k}(0)=0$, and as before, one has that $n_{k}(\infty)=0$. Moreover,
using that $\det m(\nu)$ is negative between $\nu_{-}$ and $\nu_{+}$, then
$n_{k}(\nu)=1$ for $\nu\in(\nu_{-},\nu_{+})$. Thus, we conclude that $\eta
(\nu_{-})=-1$ and $\eta(\nu_{+})=1.$
\end{proof}

Since $\omega_{n}=0$, then $\nu_{-}=0$ for $k=n$. Thus, in the following
theorem there is no bifurcation from $2\pi/\nu_{-}$, for $k=n$.

\begin{theorem}
For each $k\in\{2,...,n-2,n\}$ such that $\omega<\omega_{k}$, the polygonal
equilibrium has a global bifurcation of traveling waves from $2\pi/\nu_{+}$,
for each $\gamma\in\mathbb{R}$, with symmetries $\mathbb{\tilde{Z}}_{n}(k)$.
For each $k\in\{2,...,n-2,n\}$ such that $\omega>\omega_{k}$, there are two
bifurcations from $2\pi/\nu_{+}$ and $2\pi/\nu_{-}$, for each $\gamma
^{2}>\omega/2+(\omega_{k}(\omega-\omega_{k}))^{1/2}$, with symmetries
$\mathbb{\tilde{Z}}_{n}(k)$.
\end{theorem}

The bifurcating branches cannot return to the polygonal equilibrium for
$\omega<\omega_{k}$.

\subsubsection*{Blocks $k\in\{1,n-1\}$}

Here the block $m_{1}(\nu)$ is%
\[
m_{1}(\nu)=\left(
\begin{array}
[c]{ccc}%
\mu\left(  \mu\nu^{2}-2\gamma\nu+s_{1}+\mu\right)  & -\left(  n/2\right)
^{1/2}\mu & -\left(  n/2\right)  ^{1/2}\mu i\\
-\left(  n/2\right)  ^{1/2}\mu & \nu^{2}+s_{1}+2\mu & 2\gamma\nu i\\
\left(  n/2\right)  ^{1/2}\mu i & -2\gamma\nu i & \nu^{2}+s_{1}%
\end{array}
\right)  \text{.}%
\]
The determinant of this matrix is a polynomial of degree six in $\nu$, without
an explicit factorization. So, analytically one is able to analyze only some
special cases.

For $\mu\in(-s_{1},0)\cup(s_{1}^{2},\infty)$, the determinant at $\nu=0$ is
negative,%
\[
d_{1}(0)=-\mu(\mu+s_{1})(\mu-s_{1}^{2})<0.
\]
Moreover, since the determinant $d_{1}(\nu)$ is positive for big $\left\vert
\nu\right\vert $ , then the matrix $m_{1}(\nu_{k})$ must change its Morse
index at least at some values $\nu_{1}$ and $-\nu_{n-1}$. Thus, one has that
$\eta_{k}(\nu_{k})\neq0$ for $k\in\{1,n-1\}$.

\begin{theorem}
For each $k\in\{1,n-1\}$ such that $\mu\in(-s_{1},0)\cup(s_{1}^{2},\infty)$,
the polygonal relative equilibrium has at least one global bifurcation of
traveling waves for each velocity $\gamma\in\mathbb{R}$, starting from the
period $2\pi/\nu_{k}$, and with symmetries $\mathbb{\tilde{Z}}_{n}(k)$.
\end{theorem}

Now, we analyze completely the particular case $\mu=1$, that is when all the
$n+1$ filaments have the same circulation, with%
\[
\omega=s_{1}+1.
\]

\begin{proposition}
Define $\bar{\nu}_{\pm}$ and $\nu_{\pm}$ as%
\[
\bar{\nu}_{\pm}=\gamma\pm\left(  \gamma^{2}-\omega\right)  ^{1/2}\text{ and
}\nu_{\pm}=\left(  -b \pm\sqrt{b^{2}-c}\right)  ^{1/2}\text{,}%
\]
where $b=\omega-2\gamma^{2}$ and $c=\omega^{2}-2\omega>0$.

\begin{description}
\item[(a)] For $(\omega+\sqrt{c})/2<\gamma^{2}<\omega$, the matrix $m_{1}%
(\nu)$ changes its Morse index at $\pm\nu_{\pm}\, $with
\[
\eta_{1}(\nu_{\pm})=\pm1\emph{.}%
\]

\item[(b)] For $\omega<\gamma^{2}$, the matrix $m_{1}(\nu)$ changes its Morse
index at $\bar{\nu}_{\pm}$ and $\pm\nu_{\pm}$, with%
\[
\eta_{1}(\bar{\nu}_{\pm})=\pm1\text{ and }\eta_{1}(\nu_{\pm})=\pm1\emph{.}%
\]

\end{description}
\end{proposition}

\begin{proof}
For $\mu=1$ one finds that the eigenvalues of $m_{1}(\nu)$ are
\[
\lambda_{0}=\nu^{2}-2\gamma\nu+\omega\text{ and } \lambda_{\pm}=\nu^{2}%
+\omega\pm\sqrt{4\gamma^{2}\nu^{2}+2\omega}\text{.}%
\]
The eigenvalue $\lambda_{0}$ is zero only at $\bar{\nu}_{\pm}$ for $\gamma
^{2}>\omega$, and the eigenvalue $\lambda_{+}$ remains always positive.

The eigenvalue $\lambda_{-}$ is zero at the solutions of $\nu^{4}+2b\nu
^{2}+c=0$, that is at $\nu^{2}=-b\pm\sqrt{b^{2}-c}$. Since $c$ is positive for
$n\geq3$, then $\lambda_{-}$ is zero at $\pm\nu_{\pm}$ only if $b<0$ and
$b^{2}-c>0$. For $\gamma^{2}>(\omega+\sqrt{c})/2$, one has that $b<0$ and
$b^{2}-c=(2\gamma^{2}-\omega)^{2}-c>0$, then the eigenvalue $\lambda_{-}$ is
zero at $\pm\nu_{\pm}$ only for $\gamma^{2}>(\omega+\sqrt{c})/2$.

Since $\omega=s_{1}+1$, then $\sigma=sgn(\omega)=1$. For (a), the eigenvalues
$\lambda_{0}$ and $\lambda_{+}$ are positive, and $\lambda_{-}$ is negative
only when $\left\vert \nu\right\vert \in(\nu_{-},\nu_{+})$. Therefore, one has
that $\eta_{1}(\nu_{-})=-1$ and $\eta_{1}(\nu_{+})=1\emph{.}$

For (b), one has that $\lambda_{+}$ is positive, $\lambda_{0}$ is negative
when $\nu\in(\bar{\nu}_{-},\bar{\nu}_{+})$, and $\lambda_{-}$ is negative when
$\left\vert \nu\right\vert \in(\nu_{-},\nu_{+})$. Since
\[
\lambda_{-}(\bar{\nu}_{\pm})=2\gamma\bar{\nu}_{\pm}-\sqrt{4\gamma\bar{\nu
}_{\pm}^{2}+2\omega},
\]
then $\lambda_{-}(\bar{\nu}_{\pm})$ is negative and $(\bar{\nu}_{-},\bar{\nu
}_{+})\subset(\nu_{-},\nu_{+})$. Therefore, one has that $\eta_{1}(\nu
_{-})=0-1$, $\eta_{1}(\bar{\nu}_{-})=1-2$, $\eta_{1}(\bar{\nu}_{+})=2-1$ and
$\eta_{1}(\nu_{+})=1-0$\emph{.}
\end{proof}

\begin{theorem}
For $k\in\{1,n-1\}$ and $\mu=1$, the polygonal equilibrium has two global
bifurcation of periodic traveling waves with velocity $\gamma^{2}%
>(\omega+\sqrt{c})/2$, starting from the periods $2\pi/\nu_{+}$ and $2\pi
/\nu_{-}$, and with symmetries $\mathbb{\tilde{Z}}_{n}(k)$. In addition, there
are two global bifurcation from $2\pi/\bar{\nu}_{+}$ and $2\pi/\bar{\nu}_{-}$,
for each $\gamma^{2}>\omega$, with symmetries $\mathbb{\tilde{Z}}_{n}(1)$.
\end{theorem}

For $\mu=0$, the block $m_{1}(\nu)$ is given by%
\[
m_{1}(\nu)=\left(
\begin{array}
[c]{cc}%
\nu^{2}+s_{1} & 2\gamma\nu i\\
-2\gamma\nu i & \nu^{2}+s_{1}%
\end{array}
\right)  \text{.}%
\]
Setting $m_{1}(\nu)=\nu^{2}-2\gamma\nu iJ+2diag\left(  \omega-\bar{\omega}%
_{1},\bar{\omega}_{1}\right)  $, with $\omega=s_{1}$ and $\bar{\omega}%
_{1}=s_{1}/2$, one may analyze the block as we did for the blocks $m_{k}(\nu
)$. In this way one gets that $m_{1}(\nu)$ changes its Morse index at $\pm
\nu_{\pm}$ for $\gamma^{2}>s_{1}$, with $\eta_{k}(\nu_{\pm})=\pm1$, where%
\[
\nu_{\pm}^{2}=(2\gamma^{2}-s_{1})\pm\sqrt{4\gamma^{2}\left(  \gamma^{2}%
-s_{1}\right)  }\text{.}%
\]

Therefore, for $\mu=0$ and $k\in\{1,n-1\}$, the polygonal equilibrium has two
global bifurcations of traveling waves with velocity $\gamma^{2}>s_{1}$,
starting from the periods $2\pi/\nu_{+}$ and $2\pi/\nu_{-}$, and with
symmetries $\mathbb{\tilde{Z}}_{n}(k)$. These bifurcations are non-admissible
or go to another bifurcation point.

\subsection{The case $n=2$}

The irreducible representations for $n=2$ are different from those for larger
$n$, due to the action of $\mathbb{Z}_{2}$. This fact gives a change in the
representation for $k=1$, but the change of variables for $k=2$ remains the
same as for larger $n$'s. In particular the bifurcation results for the case
$k=n=2$ are similar to the ones already given.

For $k=1$, define the isomorphism $T_{1}:\mathbb{C}^{4}\rightarrow W_{1}$ as%
\begin{align*}
T_{1}(v,w)  &  =(v, 2^{-1/2}w, 2^{-1/2}w)\text{with}\\
W_{1}  &  =\{(v,w,w):v,w\in\mathbb{C}^{2}\}\text{.}%
\end{align*}

It is not difficult to compute the Hessian of the potential $V$. The matrix
$B_{2}$ is the same as before, but $B_{1}$ is now a $4\times4$ complex matrix.
One has that $s_{1}=1/2$.

\subsubsection{Vortices}

The matrix $m_{1}(\nu)$ is%

\[
\left(
\begin{array}
[c]{cccc}%
\mu\left(  \mu+5/2\right)  & -i\nu\mu & -\sqrt{2}\mu & 0\\
i\nu\mu & \mu\left(  \mu-3/2\right)  & 0 & \sqrt{2}\mu\\
-\sqrt{2}\mu & 0 & 2\mu+1/2 & -i\nu\\
0 & \sqrt{2}\mu & i\nu & 1/2
\end{array}
\right)  \text{.}%
\]%
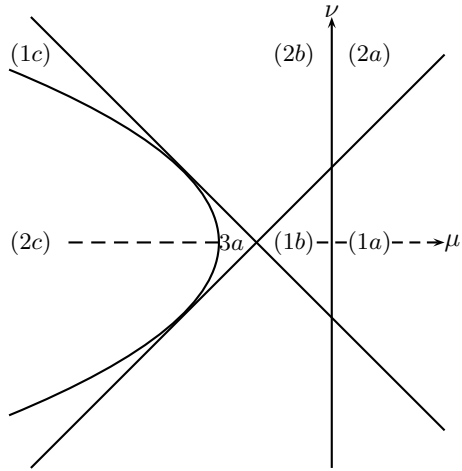
\begin{figure}[ht]
\centering\subfloat{\begin{pspicture}(-4.5,-3)(1.5,3)\SpecialCoor
\psline[linestyle=dashed]{->}(.8,0)(1.5,0)
\psline[linestyle=dashed](-3.5,0)(-1.5,0)
\psline[linestyle=dashed](-.2,0)(.2,0)
\rput[l](1.5,0){$\mu$}
\rput[b](0,3){$\nu$}
\psline{->}(0,-3)(0,3)
\psline(-4,-3)(1.5,2.5)
\psline(-4,3)(1.5,-2.5)
\rput{90}{\parabola(2.3,4.29)(0,1.5)}
\rput(.5,2.5){\small$(2a)$}
\rput(-.5,2.5){\small$(2b)$}
\rput(-4,0){\small$(2c)$}
\rput(-.5,0){\small$(1b)$}
\rput(-4,2.5){\small$(1c)$}
\rput(.5,0){\small$(1a)$}
\rput[l](-1.5,0){\small$3a$}
\NormalCoor\end{pspicture}}
\caption{Graph of $d_{1}(\mu,\nu)=0$.}
\end{figure}

\begin{proposition}
Let%
\[
\text{ }\nu_{0}(\mu)=\left\vert \mu+1/2\right\vert \text{ and }\nu_{1}%
(\mu)=\sqrt{3}(-\mu-5/4)^{1/2}\text{,}%
\]
then the determinant of $m_{1}(\nu)$ is zero only on the curves: $\mu=0$ for
$\nu\in\mathbb{R}$, $\pm\nu_{0}$ for $\mu\in\mathbb{R}$ and $\pm\nu_{1}$ for
$\mu\in(-\infty,-5/4)$. Furthermore, dividing the semiplane $\nu>0$ in the
seven regions separated by these curves, the Morse number of $m_{1}(\nu)$ in
the regions (1x) is $n_{1}=1$, in the regions (2x) it is $n_{1}=2$ and in the
regions (3x) it is $n_{1}=3$.
\end{proposition}

\begin{proof}
Since the determinant $\det m_{1}(\nu)$ is
\[
\det m_{1}=\mu^{2}\left(  \nu^{2}-(\mu+1/2\right)  ^{2})\left(  \nu^{2}%
+3(\mu+5/4\right)  )\text{,}%
\]
then this determinant is zero on the above curves, which intersect only at
$\mu=-2$.

For $\nu$ large, the Morse number of $m_{1}(\nu)$ is the same as the one of
the matrix $-\nu diag(\mu iJ, iJ)$. Hence, $n_{1}=2$ in the regions (2a) and
(2b). Now, the matrix $m_{1}(\mu,0)$ has eigenvalues%

\begin{align*}
&  \frac{1}{4}(2\mu^{2}-3\mu+1)\pm\frac{1}{4}\sqrt{4\mu^{4}-12\mu^{3}%
+37\mu^{2}+6\mu+1}\text{ and}\\
&  \frac{1}{4}(2\mu^{2}+9\mu+1)\pm\frac{1}{4}\sqrt{4\mu^{4}+4\mu^{3}+29\mu
^{2}-2\mu+1}.
\end{align*}

Then, $n_{1}=2$ in (2c), $n_{1}=3$ in (3a) and $n_{1}=1$ in (1a) and (1b).

In order to find the Morse number in (1c), one sees that the matrix $m_{1}%
(\mu, \mu+1/2)$ has eigenvalues $0$, $3\mu$ and%

\[
\mu^{2}+1/2\pm\sqrt{(\mu^{2}+1/2)^{2}+\mu\left(  \mu+2\right)  \left(
2\mu+1\right)  }.
\]

Since $m_{1}(\mu, \mu+1/2)$ has two positive eigenvalues for $\mu<-2$, then
$n_{1}\leq2$ in (1c). But, since $det(m_{1})$ is negative in (1c), then
$n_{1}=1$ there.

Note that the polygonal equilibrium is stable for $\mu<-5/4$.
\end{proof}

Thus, there is a global bifurcation of periodic orbits, starting from the
period $2\pi/{\nu_{0}}$ and, for $\mu<-5/4$, starting form the period
$2\pi/{\nu_{1}}$, and symmetry $\mathbb{\tilde{Z}}_{1}(1)$, that is
$u_{0}(t+\pi)=-u_{0}(t)$ for the central element, and $u_{2}(t)=-u_{1}(t+\pi)$
for the two elements with circulation $1$.

\subsubsection{Filaments}

The matrix $m_{1}(\nu)$ is%

\[
\left(
\begin{array}
[c]{cccc}%
\mu\left(  \mu\nu^{2}+\mu+5/2\right)  & -2i\nu\gamma\mu & -\sqrt{2}\mu & 0\\
2i\nu\gamma\mu & \mu\left(  \mu\nu^{2}+\mu-3/2\right)  & 0 & \sqrt{2}\mu\\
-\sqrt{2}\mu & 0 & \nu^{2}+2\mu+1/2 & -2i\nu\gamma\\
0 & \sqrt{2}\mu & 2i\nu\gamma & \nu^{2}+1/2
\end{array}
\right)  .
\]

In this case the determinant is a polynomial of degree $8$, with no easy
factorization. In the simple case of $\nu=0$, the determinant is%

\[
\det m_{1}(0)=-3\mu^{2}\left(  \mu+1/2\right)  ^{2}\left(  \mu+5/4\right)
\text{,}%
\]

Then, the determinant $\det m_{1}(\mu,0)$ is negative for $\mu>-5/4$, but
$\det m_{1}(\mu,\nu)$ is positive for $\nu$ large. Thus, for $\mu>-5/4$, there
is a value $\nu_{1}$ where $\det m_{1}(\mu,\nu)$ changes sign. Hence,
$\eta_{1}(\nu_{1})\neq0$ for $\mu>-5/4$. That is, for $\mu>-5/4$ and any
$\gamma$, there is a global bifurcation of periodic traveling waves starting
at the period $2\pi/{\nu_{1}}$, wave velocity $\gamma$ and symmetry
$\mathbb{\tilde{Z}}_{1}(1)$, as before.


\end{document}